\documentclass{article}

\usepackage[margin=1.5in]{geometry}
\usepackage{amsmath, amsfonts, amssymb, amsthm}
\usepackage{graphicx}
\usepackage{mdwlist}
\usepackage[all]{xy}
\newtheorem{theorem}{Theorem}[section]
\newtheorem{proposition}[theorem]{Proposition}
\newtheorem{definition}[theorem]{Definition}
\newtheorem{rmk}[theorem]{Remark}
\newtheorem{corollary}[theorem]{Corollary}
\newtheorem{conjecture}[theorem]{Conjecture}
\newtheorem{problem}[theorem]{Problem}

\newenvironment{remark}{\begin{rmk}\rm}{\end{rmk}}

\newcommand{\Z}{\mathbf{Z}}
\newcommand{\des}{\operatorname{des}}

\newcommand{\inv}{\operatorname{inv}}
\newcommand{\T}{\mathcal{T}}

\author{Joel Brewster Lewis \\ University of Minnesota \and Nan Li \\ MIT }

\title{Flashcard games}

\begin{document}
\maketitle
\begin{abstract}
We study a family of discrete dynamical processes introduced by Novikoff, Kleinberg and Strogatz that we call \emph{flashcard games}. We prove a number of results on the evolution of these games, and in particular we settle a conjecture of NKS on the frequency with which a given card appears. We introduce a number of generalizations and variations that we believe are of interest, and we provide a large number of open questions and problems.
 \end{abstract}

\section{Introduction}

In their paper \cite{NKS}, Novikoff, Kleinberg and Strogatz introduced a combinatorial process that we will call a \emph{flashcard game}.  These games are defined as follows: as initial data, we have a sequence $(p_k)_{k \in \Z_{>0}}$ called the \emph{insertion sequence}, and a \emph{deck} of infinitely many \emph{cards} $1, 2, 3, \ldots$.  For each $t \geq 1$, at time $t$ we look at the first card in the deck; if we are looking at it for the $k$th time, we remove it and insert it into position $p_k$.  For example, with $p_k = k + 1$, the procedure evolves as follows: at time $t = 1$ we see card $1$ for the first time, after which we insert it into the deck in position $p_1 = 2$, leaving the deck in the order $2, 1, 3, 4, \ldots$.  At time $t = 2$, we now see card $2$ for the first time, so we insert it into position $p_1 = 2$ to return the deck to the order $1, 2, 3, \ldots$.  At time $t = 3$, we see card $1$ for the second time, so we insert it into position $p_2 = 3$, leaving the deck in the order $2, 3, 1, 4, \ldots$.  And so on.  Novikoff \emph{et al.} suggest that such processes may be used as a model of student attempts to memorize a growing list of information; however, flashcard games also have substantial appeal as pretty but complicated examples of discrete dynamical systems.

In this paper, we expand the study of flashcard games.   In Section~\ref{sec:slow}, we settle a conjecture of Novikoff \emph{et al.} on the frequency with which cards appear at the front of the deck; in particular, we show that when $p_k = k + 1$, the time until the $n$th viewing of a given card $i$ grows like a second-degree polynomial in $n$.  We also prove a variety of other results on the behavior of the function $T_i(n)$ that gives the time of the $n$th viewing of the $i$th card.  In Section~\ref{sec:other objects}, we introduce several objects not considered by Novikoff \emph{et al.} related to flashcard games, and we establish a number of connections between these objects.  Most interestingly, we conjecture the existence of a curve that describes the long-term time evolution of the flashcard game.  In Section~\ref{sec:generalizing}, we extend most of the results of the preceding sections to general insertion sequences $(p_k)$.  In Section~\ref{sec:open questions}, we suggest several other generalizations, variations, and open problems that may be of interest.

\section{Definitions and notation}

Given a sequence $(p_k)_{k \in \Z_{> 0}}$ of positive integers, we define a discrete dynamical process as follows: the state consists of a permutation of the positive integers (the \emph{deck}) together with a counter that records how many times each integer (a \emph{card}) has been at the front of the deck; thus, the initial state consists of the permutation $(1, 2, 3, 4, \ldots)$, the card $1$ has been seen once, and all other cards have been seen $0$ times.  Every subsequent state follows from the state that precedes it by moving the card at the front of the deck to position $p_k$, where $k$ is the number of times it has been seen so far, and incrementing the counter for the card now at the front of the deck.

Following \cite{NKS}, we call the flashcard game with insertion sequence $(2, 3, 4, 5, \ldots)$ the \emph{Slow Flashcard Game}.  In this section, we define several pieces o terminology and notation for this game; in later sections, we will continue to use this notation but in a more general setting.  There are several possible choices for clock behavior for a flashcard game.  We choose the following one: at time $t = 1$, card $1$ is in front of the deck and has been viewed once.  We then move card $1$ to the second position in the deck; it is now time $t = 2$ and we are looking at card $2$.  For $n, k \geq 1$ we denote by $T_n(k)$ the time we see card $n$ for the $k$th time, so we have $T_1(1) = 1$, $T_2(1) = 2$, $T_1(2) = 3$, and so on.  In particular, the sequence $(T_1(k))_{k \geq 1}$ marks those times when we see card $1$, and the sequence $(T_n(1))_{n \geq 1}$ marks those times when we see a new card for the first time.

The sequence $1, 2, 1, 2, 3, 1, 3, \ldots$ of cards seen at times $t = 1, 2, \ldots$ is called the \emph{viewing sequence} of the flashcard game.

Define $c_n(t)$ to be the number of times card $n$ has been seen at time $t$.  Thus, we have that $c_n(t) = k$ exactly when $T_n(k) \leq t < T_n(k + 1)$, and also $\sum_n c_n(t) = t$.

\section{The Slow Flashcard Game}\label{sec:slow}

In the first part of this section, we make a few simple observations about the dynamics of the Slow Flashcard Game and of flashcard games in general.  In later subsections, we prove some nontrivial results, including the resolution of a conjecture of \cite{NKS}.  The first few observations are essentially trivial; we collect them in a single proposition.  (Parts of this proposition were dubbed the ``no passing property'' and ``slow marching property'' in \cite{NKS}, but we don't need these names here.)

\begin{proposition}\label{basic}
\begin{enumerate*}
\item When the card at the front of the deck is inserted into position $m$, the cards previously in positions $2$, $3$, \ldots, $m$ move forward one position, and all other cards remain fixed.
\item For $i \geq 1$, card $i$ remains in position $i$ until the first time a card is inserted in position $m$ for some $m \geq i$.
\item Fix a time $t$, and let $t_n$ be the smallest time such that $t_n \geq t$ and $n$ is at the front of the deck at time $t_n$.  If card $i$ is at position $m$ at time $t$, then $t_i \geq t + m - 1$.
\item If card $i$ precedes card $j$ in the deck at time $t$ then $t_i < t_j$.
\item If $i < j$ then $c_i(t) \geq c_j(t)$ for all $t$.
\end{enumerate*}
\end{proposition}

Parts 1 and 2 are obvious.  Part 3 says that cards cannot move more than one space forward in the deck at each time-step.  Parts 4 and 5 are the observations that cards can't jump over cards that are in front of them in the deck, and so in particular no card may be seen more often than a smaller-numbered card.

Next, we show how the functions $T_1(k)$ and $T_n(1)$ interleave with each other by studying the first time each card is seen; this is a slight strengthening of \cite[Theorem 8]{NKS}.

\begin{theorem}\label{1andi}
For any integer $n \geq 1$ we have
\[
T_1(i) + i - 1 \leq T_i(1) < T_1(i).
\]
\end{theorem}
\begin{proof}
For $i = 2$ the result is immediate.  For $i > 2$, at time $T_1(i-1) + 1$, the deck has the form
\[
\begin{array}{rccccc}
  \text{Positions:} &\cdots & i-1       & i         & i+1         & \cdots \\
  \text{Cards:}     &\cdots & \boxed{i} & \boxed{1} & \boxed{i+1} & \cdots \\
  \text{Times seen:}&       & 0         & i - 1     & 0           &
\end{array}
\]
and card $i$ has just moved to position $i - 1$ for the first time, so in particular has not been seen yet.  It follows from Proposition \ref{basic}, part 3, that $T_i(1) \geq T_1(i-1) + i - 1$.

At time $T_1(i-1) + 1$, card $1$ is behind card $i$ in the deck until card $i$ is seen and then reinserted after card $1$. By Proposition~\ref{basic}, part 4, we have $T_1(i)>T_i(1)$.
\end{proof}

\begin{corollary}\label{cor:adding k to both sides}
For any positive integers $i$ and $k$ we have $T_i(1+k) < T_1(i + k)$.
\end{corollary}
\begin{proof}
By Theorem~\ref{1andi}, at any time, card $1$ has been seen no fewer times than card $i$.  Thus, after time $T_1(i - 1)$, every time card $1$ is seen it jumps behind card $i$ in the deck.  It follows that card $i$ is seen at least once between consecutive viewings of card $1$.  Thus, for any $k$, card $i$ will be seen for the $(k+1)$th time before card $1$ is for the $(i - 1) + (k+1)$th time.  The result follows.
\end{proof}

It is not possible to add $k$ in a similar way to the other half of the inequality in Theorem~\ref{1andi}.  In fact, our data suggests that for any pair of cards, eventually the numbers of times the two cards have been be seen will converge.  We make this precise in the following proposition and conjecture.
\begin{proposition}\label{after}
For any cards $i$ and $j$, suppose that at time $t$ we have $c_i(t) = c_j(t) > 0$, i.e., cards $i$ and $j$ have been seen the same (positive) number of times.  Then for all $t' > t$, we have $|c_i(t') - c_j(t')| \leq 1$.
\end{proposition}
The proof of the proposition is straightforward and we omit it here.
\begin{conjecture}\label{stable}
For any card $i$, there exists $t$ such that $c_i(t)=c_1(t)$.
\end{conjecture}
Thus, for fixed $n$ we expect that after some sufficiently large time, the cards $1$ through $n$ will all have been seen the same number of times.  After this (conjectural) time, the dynamics of these cards are trivial.  It seems interesting to try to compute good bounds for this time.

\subsection{New cards are seen at quadratic rate}

In this section, we settle Conjecture 1 of \cite{NKS}; that is, we show that the functions $T_1(n)$ and $T_n(1)$ have growth rate $\Theta(n^2)$.

\begin{theorem}\label{T_1(n) is quadratic}
For all $n \geq 1$ we have $T_1(n)\le n^2 - n + 1$.
\end{theorem}
\begin{proof}
The result is clearly true for $n = 1,2$.  For $n > 2$, let $t = T_1(n - 1) + 1$.  At time $t$, the deck has the form
\[
\begin{array}{rccccc}
  \text{Positions:}& \cdots & n-1 & n & n+1 & \cdots \\
  \text{Cards:}&\cdots & \boxed{n} & \boxed{1} & \boxed{n+1} & \cdots \\
  \text{Times seen:}&  &  0 & n - 1 & 0
\end{array},
\]
the cards preceding $1$ are exactly those in the set $A=\{2,3,\dots,n\}$, and these include all the cards that have been seen so far.  Recall that $c_i(t)$ denotes the number of times that card $i$ has been seen at this moment.  Then we have $c_2(t) + c_3(t) + \cdots + c_n(t) = t - n + 1$.

Before we see card $1$ again, each of the cards in $A$ must be inserted behind card $1$.  For each card $i \in A$, we have that card $i$ will next be inserted behind card $1$ no later than the $(n - 1)$st viewing of card $i$.  At time $t$, card $i \in A$ has already been seen $c_i(t)$ times, so we need to see it at most $n - c_i(t) - 1$ more times before the next time we see card $1$.  Summing over all cards in $A$, we have
\[
T_1(n) - t \le 1 + \sum_{i=2}^n (n - c_i(t) - 1) = n^2 - n - t + 1,
\]
which completes the proof.
\end{proof}

From Theorem~\ref{1andi} it follows immediately that $T_1(n) \geq \binom{n + 1}{2}$ and so also that $T_1(n) \geq \frac{n^2}{2} + O(n)$.  Together with Theorem~\ref{T_1(n) is quadratic}, this suggests that actually $T_1(n) \sim c \cdot n^2$ for some constant $c \in [1/2, 1]$.  Numerical experiments suggest the following conjecture.

\begin{conjecture}\label{quadratic}
We have $T_1(n) \sim c n^2$ for $c \approx 0.85\ldots$.
\end{conjecture}

This agrees with the numerical data in Figure S3 of \cite{NKS}.  We remark that unfortunately our work provides no improvement in the bounds on the differences $T_i(n + 1) - T_i(n)$, so the following intriguing conjecture of Novikoff \emph{et al.} is still quite open.

\begin{conjecture}[{\cite[Conjecture 2]{NKS}}]\label{conj:NKS}
We have $T_i(n + 1) - T_i(n) \leq 2n$ for all $i$ and $n$.
\end{conjecture}

Mark Lipson kindly provided the following result closely related to Theorem~\ref{T_1(n) is quadratic} (private communication following discussion with the audience at the MIT graduate student seminar SPAMS):
\begin{theorem}\label{T_n(1) is quadratic}
We have $T_n(1) \leq (n - 1)^2 + 1$.
\end{theorem}
\begin{proof}
We have by Theorem \ref{1andi} that $T_n(1)<T_1(n)$. Therefore, at time $T_n(1)$, card $n$ is being seen for the first time, cards $1$ through $n - 1$ have each been seen at most $n - 1$ times, and no other cards have been seen.  Thus, at the earliest this happens at time $(n - 1)^2 + 1$.
\end{proof}

\subsection{Other results about the Slow Flashcard Game}

In this subsection, we seek to extend our knowledge about the function $T_n(k)$.  The first result improves on the naive result of Corollary~\ref{cor:adding k to both sides}, and may be viewed as a first attempt in the direction of Conjecture~\ref{stable}.

\begin{theorem}\label{root2k}
For all $k \geq 1$ and all $\ell \leq \sqrt{2k}+O(1)$ we have $T_{k}(\ell) < T_1(k + 1)$.
\end{theorem}
\begin{proof}
At time $T_1(k)+1$, the deck has the form
\[
\begin{array}{rcccccc}
  \text{Positions:} &\cdots & a         & \cdots & k           & k+1       & \cdots \\
  \text{Cards:}     &\cdots & \boxed{k} & \cdots & \boxed{k+1} & \boxed{1} & \cdots \\
  \text{Times seen:}&       & b         &        & 0           & k         &
\end{array}
\]
where card $1$ is in the $(k+1)$th position.  Assume card $k$ is in the $a$th position and has been seen $b$ times.  Note that by Theorem~\ref{1andi}, $b > 0$. Moreover, we have $0 < a \le b$. Before card $1$ is seen again, card $k$ must jump over all the cards between card $k$ and card $1$, as well as card $1$ itself.  In its next $\ell$ jumps, card $k$ jumps over at most $(b+1)+(b+2)+\cdots+(b+\ell)$ of these cards.  Thus, it must be seen at least $m$ more times, where $m$ is the minimal integer such that
\[
(b+1)+(b+2)+\cdots+(b+m) \ge k+1-a \ge k + 1 - b.
\]
So all together card $k$ is seen $b+m$ times before card $1$ is seen for the $(k+1)$th time. Simplifying the condition, we get that the minimal $m$ satisfies
\[
bm+\frac{m^2+m}{2}\ge k+1-b
\]
and so is given by the horrible formula that results from solving for the equality case; this gives that $b+m$ is $\sqrt{b^2-b+2k}+O(1)$.  The minimal possible value of this expression is $\sqrt{2k}+O(1)$ when $b=1$.  The result follows immediately.
\end{proof}

The method of the previous result can be iterated to replace $T_1(k + 1)$ with $T_1(k + i)$ (for $i$ not too large) and to replace $\sqrt{2k}$ with a correspondingly larger value.

The second result in this section may be viewed as a first step in thinking about Conjecture~\ref{conj:NKS}.

\begin{theorem}
For any $k$ and $i$ such that $\binom{i + 1}{2}<k$, we have $T_k(i + 1) - T_k(i) = i + 1$ (the minimal possible value).
\end{theorem}
\begin{proof}
%At time $T_1(k-1)+1$, the deck has the form
%\[
%\begin{array}{rccccc}
%  \text{Cards:  }& \cdots &\boxed{k} & \boxed{1} & \boxed{k+1} &  \cdots \\
%  \text{Positions:  }&   \cdots & k-1    &  k      & k+1             &\cdots
%\end{array}
%\]
%where $1$ is in the $k$th position.
By Theorem \ref{1andi}, at time $T_k(1)$ card $1$ has not been seen the $k$th time, and therefore card $k+1$ has not moved yet.  Thus, at time $T_k(1)$ the deck has the form
\[
\begin{array}{rcccc}
  \text{Positions:  }&    1  & \cdots & k+1         &  \cdots \\
  \text{Cards:  }& \boxed{k} & \cdots & \boxed{k+1} &  \cdots \\
  \text{Times seen:  }&  1   &        & 0           &
\end{array}
\]
where $k+1$ is in the $(k+1)$th position and has not been seen yet, and the cards in positions $2$ through $k$ are exactly the cards with numbers $1$, $2$, \ldots, $k - 1$.  %At each time step after $T_k(1)$ until the time at which card $k$ is first inserted behind card $k + 1$, the distance between cards $k$ and $k+1$ weakly decreases.
Choose any $i$ such that at time $T_k(i) + 1$, card $k$ is still in front of card $k+1$.  For each card $j$ in front of card $k$, we have $j<k$.  By part 5 of Proposition~\ref{basic}, card $j$ has been seen at least as many times as card $k$, so after we next see card $j$ we insert it following card $k$ in the deck.  Thus, on each time-step card $k$ moves forward by one position, and it follows that $T_k(i + 1) - T_k(i) = i + 1$ as long as card $k$ is still in front of card $k+1$.

Now we estimate the time when card $k$ is inserted after card $k+1$. Notice that at time $T_k(1)$, the distance between card $k$ and $k+1$ is $k$, and each time card $k$ is seen again, say the $m$th time, this distance is shortened by at most $m$. Therefore, at time $T_k(i+1)$, for any $i$ such that
$1+2+\cdots+i<k$, we know that card $k$ is still in front of card $k+1$. Therefore, for all such $i$ we have $T_k(i + 1) - T_k(i) = i + 1$, as claimed.
\end{proof}

\section{Other combinatorial objects}\label{sec:other objects}
In this section, we introduce and study some other objects associated to the flashcard game.
\subsection{Viewing sequence and counting sequence}
There are two sequences naturally associated to the flashcard game.  The first is the \emph{viewing sequence}, defined earlier, whose $t$th term $V_t$ records which card we see at time $t$.  The second is the \emph{counting sequence}, whose $t$th term $C_t$
records how many times we have seen the card that is visible at time $t$.  The first $30$ terms of the viewing sequence are
\[
1,2,1,2,3,1,3,2,4,3,4,1,2,4,3,5,1,5,4,2,5,3,6,4,6,5,1,6,2,3,\dots
\]
and the first $30$ terms of the counting sequence are
\[
1,1,2,2,1,3,2,3,1,3,2,4,4,3,4,1,5,2,4,5,3,5,1,5,2,4,6,3,6,6,\dots.
\]

\begin{proposition}
 The viewing sequence and counting sequence are equivalent, i.e., we can recover one from the other without going through the entire flashcard process.
\end{proposition}
\begin{proof}
To go from the viewing sequence $(V_i)$ to counting sequence $(C_i)$, we simply count $C_i=\#\{j\le i\mid V_j=V_i\}$.  For the other direction, it is not hard to construct the viewing sequence from the counting sequence based on the
following two simple observations:
\begin{enumerate*}
\item by part 5 of Proposition \ref{basic}, we see card $i$ no less often than card $j$ if and only if $i<j$, and
\item as time increases, the number of times we see any particular card increases.
\end{enumerate*}
Then for each $k$, look at the subsequence of the counting sequence $b_{i_1},b_{i_2},\dots$ with all $b_{i_j}=k$. Let $a_{i_j}=j$. In other words, Label each occurrence of $k$ from left to right with the numbers $1, 2, 3, \dots$. Do this for all $k$ and we recover the viewing sequence.
\end{proof}

Given a word of combinatorial interest, one possible method of examining it is to apply the famous Robinson-Schensted-Knuth correspondence, henceforth RSK.  For background, definitions and properties of RSK, see for example \cite[Chapter 7]{EC2}.  After applying RSK to a sequence, we get a pair of semi-standard Young tableaux of the same shape; one is called the \emph{insertion tableau} and the other is called the \emph{recording tableau}.  Given an infinite sequence on $\Z_{>0}$ in which every term appears infinitely many times, one natural way to apply RSK is with the reversed order $1>2>3>\cdots$
for the insertion.\footnote{One could of course also use the usual order $1 < 2 < 3 < \ldots$, but it is not clear whether the resulting tableaux have any combinatorial significance.}  Using this ordering, the insertion tableau never stabilizes.  However, by definition of the viewing and counting sequences and the basic properties listed in the preceding proof, we have the following alternate definition of the recording tableau.  Let $\T$ be the tableau in the quarter-plane in which the box $(i, j)$ is filled with the value $T_i(j)$.\footnote{That this object $\T$ really is a tableau, \textit{i.e.}, that it increases along rows and columns, is straightforward: one set of comparisons is trivial and the other follows from Proposition~\ref{basic}, part 5.}
\begin{proposition}
Applying RSK with the reverse order $1 > 2 > \ldots$, the recording tableau for the viewing sequence is the tableau $\T$.  The recording tableau for the counting sequence is the transpose of $\T$.
\end{proposition}
\begin{proof}
Suppose we apply RSK to a finite prefix of the viewing sequence.  By Proposition~\ref{basic}, part 5, all columns of the insertion tableaux are of the form $k,k-1,\dots,2,1$. Therefore, suppose now we see card $r$ for the $m$th time. Then in the insertion tableaux there are already $m-1$ $r$'s, which are in the first $m-1$ columns. So when we insert the $m$th $r$, it bumps the top number on the $m$th column, and this number should be $r-1$. Therefore, this insertion pushes the $r-1$ numbers down and we insert $r$ in the first row of the $m$th column. As a result, the new spot is in the $m$th column and $r$th row, as recorded in the recording tableau. A very similar argument works for the counting sequence with rows and columns exchanged.
\end{proof}

% The function $T$ may be viewed as a map from $\Z_{>0}^2 \to \Z_{>0}$, or equivalently as a tableau $\T$ in the quarter-plane in which the box in the $n$th row and $k$th column is filled with $T_n(k)$.  This tableau $\T$ is increasing in rows ($T_i(n)<T_i(n+1)$), i.e., card $i$ is seen the $(n+1)$th time after it is seen the $n$th time; part 5 of Proposition~\ref{basic} below is the statement that $\T$ is also increasing in columns ($T_i(n) < T_{i+1}(n)$), i.e., card $i+1$ is seen the $n$th time after card $i$ is seen the $n$th time. Let $\T_i$ be the part of $\T$ consisting of all terms smaller than $i$. For example, the tableau $\T_{30}$ is shown below:
% \[
% \T_{30}=\begin{array}{cccccc}
%      1 &  3 &  6 & 12 & 17 & 27 \\
%      2 &  4 &  8 & 13 & 20 & 29 \\
%      5 &  7 & 10 & 15 & 22 &    \\
%      9 & 11 & 14 & 19 & 24 &    \\
%     16 & 18 & 21 & 26 &    &    \\
%     23 & 25 & 28 &    &    &
%   \end{array}.
% \]

\subsection{Two ways to describe the deck}
At time $t$, let the deck of cards be in the order $(u_1,u_2,u_3,\dots)$.  Instead of describing the state of the game in this way, we can alternatively give
the sequence $c_{u_i}(t)$, \textit{i.e.}, we can record the number
of times card $u_i$ has been seen at time $t$ (while suppressing the actual name of the card). We call this alternate representation the \emph{deck of times}.
For example, at time $t=100$, the deck of cards is
\[
4,10,7,11,5,6,8,9,12,1,2,3,13,14,15,\dots
\]
and the deck of times is
\[
10,6,9,4,10,10,9,8,0,11,11,11,0,0,\dots.
\]
\begin{proposition}
From the deck of times, we can recover the time $t$ and the deck of cards at time $t$.
\end{proposition}
\begin{proof}
In the deck of times, there are always finitely many nonzero terms, and adding them up we get $t$.  To write down the deck of cards, begin with the deck of times, decrement the first term by $1$, and choose the largest value that appears.  Suppose this value appears $m$ times; replace the appearances of this value from left to right with $1, 2, \ldots, m$.  Then choose the next-largest value that appears, and replace appearances from left to right with $m + 1, m + 2, \ldots$, and so on.  That this procedure works follows from parts 4 and 5 of Proposition~\ref{basic} (\emph{i.e.}, smaller-numbered cards are seen no fewer times than bigger-numbered cards, and for two cards that have been seen the same number of times, the smaller-numbered card appears in front of the bigger-numbered card.
\end{proof}

\begin{problem}
 At time $t$, given the deck of cards, what can we tell about the deck of times (more easily than running the insertion from scratch)?
\end{problem}
It seems very hard to construct the whole deck of times from the deck of cards and time $t$,\footnote{Note that it is not possible to reconstruct the deck of times from the deck of cards alone, for a trivial reason: the deck of cards is in the same order at time $3$ as it is at time $1$.  (It seems likely that this is the only example of such a repetition.)} but we can tell some partial information. For example, if $t > 2$ and we choose $k$ maximal so that card $k$ is not in position $k$, then $c_1(t)=k-1$.  We also know that if $c_i(t)>0$ then $c_i(t)\ge j-1$, where card $i$ is in the $j$th position of the deck.

\subsection{Limiting curve}
In this section we study the tableau $\T$ by examining the growth of $\T_i$, the finite portion of $\T$ whose entries are at most $i$, for large $i$.  It appears that for different large values of $i$, the outer boundaries of the $\T_i$ have a very similar shape. For example, the image below shows $\T_{20000} \setminus \T_{10000}$, i.e., it shows a point at position $(n, k)$ whenever $10000 < T_n(k) \leq 20000$.  The inner boundary curve is the boundary for $\T_{10000}$ and the outer boundary curve is the boundary for $\T_{20000}$.

\begin{figure}[ht]
\begin{center}
\includegraphics[height=2in]{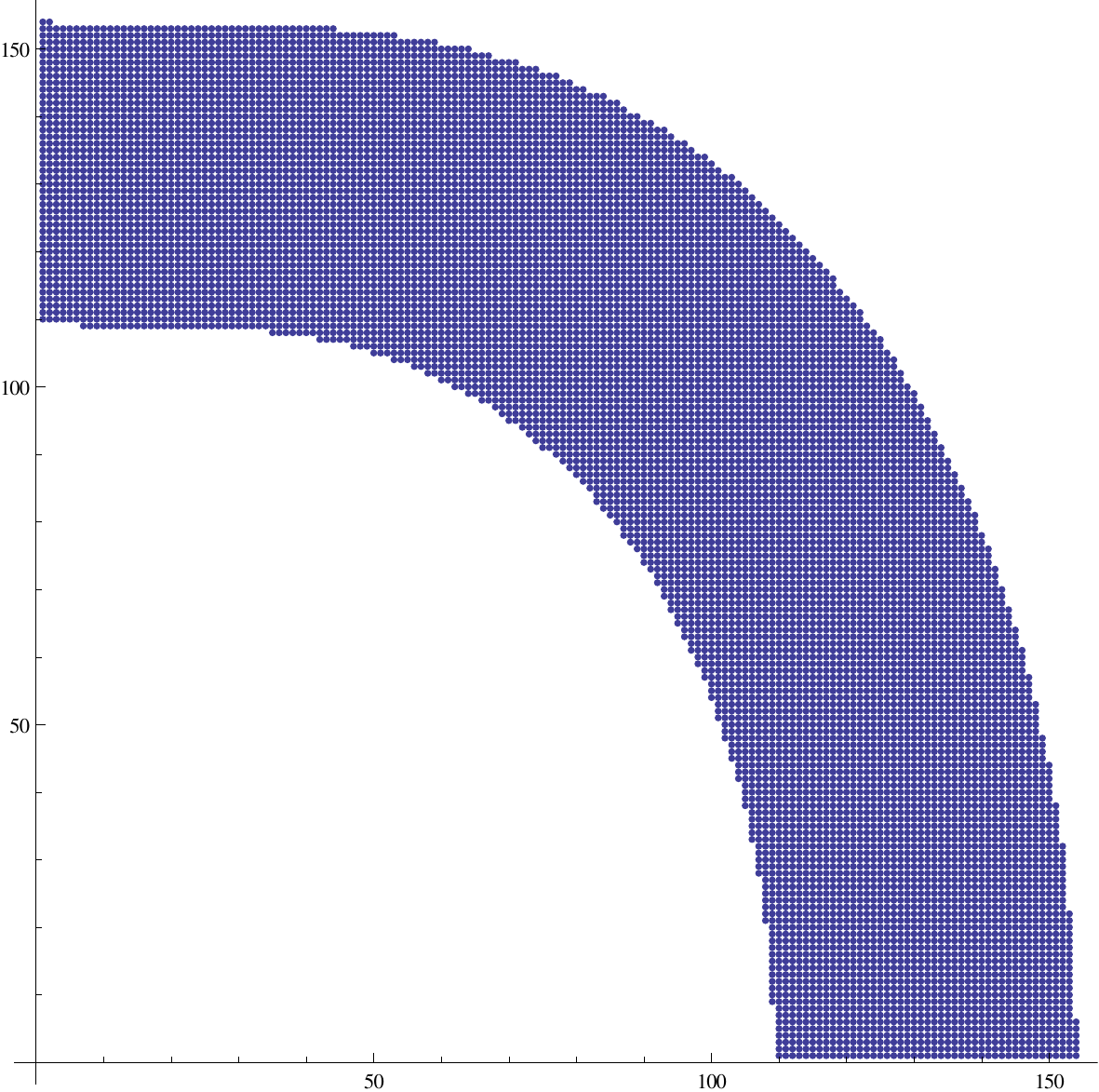}
\end{center}
\caption{The set of points $(n, k)$ such that $10000 < T_n(k) \leq 20000$; the inner and outer boundaries appear to be essentially identical up to rescaling.}
\end{figure}

An alternative way to describe this phenomenon is as follows.  Instead of plotting a dot at $(n,k)$ in the $xy$-plane for some range of values of $T_n(k)$, we plot points at positions
\[
\left( \frac{n}{\sqrt{T_n(k)}}, \frac{k}{\sqrt{T_n(k)}} \right).
\]
Given an interval $I \subset \Z_{>0}$, we denote by $A_I$ this rescaled plot of points for which $T_n(k)\in I$.  For example, Figure~\ref{fig:curve} shows the plot $A_I$ for $I=[100,10000]$.

\begin{figure}[ht]
\begin{center}
\includegraphics[height=2in]{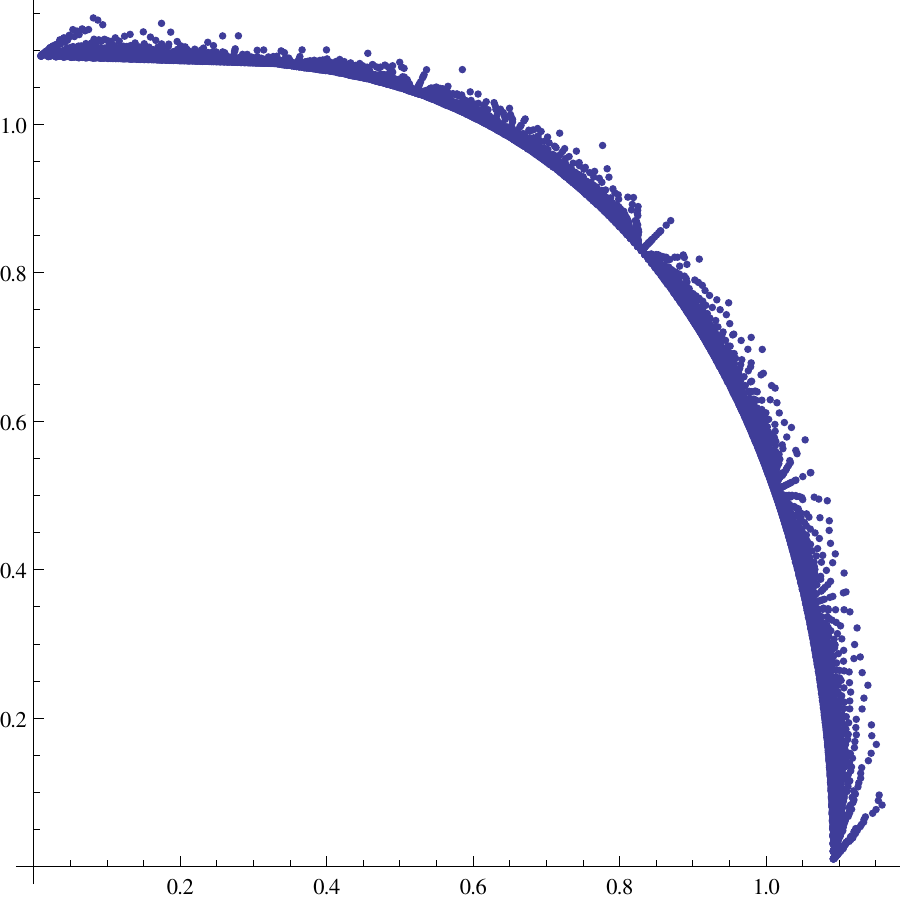}
\end{center}
\caption{The set $A_I$ for $I = [100, 10000]$, i.e., the set of all points $(n, k)/\sqrt{T_n(k)}$ for which $T_n(k) \in I$.}
\label{fig:curve}
\end{figure}

\begin{conjecture}
There exists a curve $\Gamma$ such that the area below $\Gamma$ in the first quadrant is equal to $1$ and $\frac{(n,k)}{\sqrt{T_n(k)}}$ lies outside of $\Gamma$ for all $n,k$. Moreover, as $T_n(k)$ grows larger, the point $\frac{(n,k)}{\sqrt{T_n(k)}}$ approaches $\Gamma$ in the following sense: for any $\epsilon>0$ and any $\theta\in[0,\frac{\pi}{2})$, there exists $K(\epsilon,\theta)>0$, such that for any two $k_1,k_2>K(\epsilon,\theta)$, we have
\[
|P_1-P_2|<\epsilon,
\]
where $P_i=\frac{(n_i,k_i)}{\sqrt{T_{n_i}(k_i)}}$ and $n_i=\frac{k_i}{\tan{\theta}}$ for $i=1,2$.
\end{conjecture}

We have not proved the existence of the curve, but assuming it does, we provide some nice preliminary bounds for its location.
\begin{proposition}
 All points in $A_{[1,\infty)}$ lie above the line $x+y=1$.  Also, for any $\varepsilon > 0$ and sufficiently large $M = M(\varepsilon)$, all points in $A_{[M, \infty)}$ lie below the circle $x^2+y^2 = 2 + \varepsilon$ in the first quadrant.
\end{proposition}
\begin{proof}
From Theorem~\ref{T_1(n) is quadratic} and Corollary~\ref{cor:adding k to both sides} it follows that
\[
T_n(k)<(n+k-1)^2-(n+k-1)+1=(n+k)^2-3(n+k)+3<(n+k)^2.
\]
Thus, for each point $(x, y) \in A_{[1, \infty)}$ we have
\[
x+y=\frac{n}{\sqrt{T_n(k)}}+\frac{k}{\sqrt{T_n(k)}}>\frac{n}{n+k}+\frac{k}{n+k}=1.
\]

On the other hand, we have from Theorem~\ref{1andi} (applying it successively) that
\[
T_1(n)>T_n(1)\ge \binom{n}{2}.
\]
Similar to $T_1(i)>T_1(i-1)+i-1$, we have $T_n(k)\ge T_n(k-1)+k$ and so also
\[
T_n(k)\ge \binom{n}{2}+\binom{k + 1}{2}\sim \frac{n^2+k^2}{2},
\]
from which the second half of the result follows.
\end{proof}

The study of the behavior of the plot $A$ (or the curve $\Gamma$) can tell us more information about the growth of $T_n(k)$. For example, the following result connects the curve $\Gamma$ to Conjecture~\ref{quadratic}.

\begin{proposition}
Suppose that $\Gamma$ exists and intersects the $x$-axis at the point $(c, 0)$.  Then $T_n(1) \sim n^2/c^2$.
\end{proposition}
\begin{proof}
At time $t=T_n(1)$, consider the associated point $\frac{(n, 1)}{\sqrt{T_n(1)}}$.  Let $n$ be very large.  In the limit, $\frac{1}{\sqrt{T_n(1)}}$ goes to $0$, and thus the place the curve touches the $x$-axis has $x$-coordinate $\lim_{n \to \infty} \frac{n}{\sqrt{T_n(1)}}=c$.  The result follows immediately.
\end{proof}

\section{Changing the insertion sequence}\label{sec:generalizing}
We can generalize the flashcard procedure as follows: to each sequence $(p_k)_{k \in \Z_{> 0}}$, associate the flashcard game that moves the front card to position $p_k$ when it is seen for the $k$th time.  Thus, the dynamical system studied above is the case $p_k = k + 1$, while \cite{NKS} note that their ``recap schedule'' is the case $p_k = 2^k$.

The first question of interest to \cite{NKS} is whether a flashcard schedule exhibits ``infinite perfect learning''.  In our case, this asks whether we eventually see every card (equivalently, whether we see every card infinitely often).  It turns out that this property is easy to characterize in terms of the sequence $(p_k)$.

\begin{theorem}
A sequence $(p_k)$ results in every card being seen infinitely often if and only if $(p_k)$ is unbounded.
\end{theorem}
\begin{proof}If $p_k<N$ for all $k$, then card $N$ can never move forward, so there is no chance to see it. On the other hand, if the sequence is not bounded, then for any card $i$ starting at any stage in the process, we have that eventually some card in front of $i$ will be seen sufficiently many times to be inserted after $i$.  Thus, card $i$ will eventually move to the front of the deck.  The result follows.
\end{proof}

For the rest of this section, we suppose that $(p_k)$ is unbounded.  The statistics of interest in \cite{NKS} included how often we see the first card (i.e., the growth of the function $T_1(n)$), how long it takes to see the $n$th card for the first time (i.e., the growth of the function $T_n(1)$), and how long we have to wait between instances of seeing the same card (i.e., the behavior of $T_i(n + 1) - T_i(n)$).  We now investigate these questions in our more general setting.

%\begin{problem}\label{p_k}
%Can we bound any of these in terms of the sequence $p_k$?  In particular, can we extend our arguments in the preceding section to this more general case?  [Maybe we need to assume that $(p_k)$ is increasing, or something, but this would be okay.]
%\end{problem}
%
%\begin{problem}
%Are these dynamics interesting in the case $(p_k)$ is bounded?
%\end{problem}
%
%
%
%
%We have some partial results in the direction of Problem \ref{p_k}.

If we assume that $(p_k)$ is (weakly) increasing, many results in Section~\ref{sec:slow} can be generalized to this section.

\begin{theorem}\label{thm:interweaving in general}
If $(p_k)$ is increasing then for all $n$ we have $T_1(n) + p_n - 1 \leq T_{p_n}(1) < T_1(n + 1)$ and $T_{p_{n-1}}(1+k)<T_1(n+k)$ for all $k\ge 0$.
\end{theorem}

\begin{proof}
The first chain of inequalities follows from the same argument as Theorem \ref{1andi}.

The second result is similar to Corollary \ref{cor:adding k to both sides}. At any time, card $1$ has been seen no fewer times than card $p_{i-1}$.  Thus, after time $T_1(i - 1)$, every time card $1$ is seen it jumps after card $p_{i-1}$.  It follows that card $p_{i-1}$ is seen at least once between consecutive viewings of card $1$.  Thus, for any $k\ge0$, after time $T_1(i - 1)$, card $p_{i-1}$ will be seen for the $(k+1)$st time before card $1$ is seen $k+1$ additional times.
\end{proof}

\begin{theorem}\label{thm:T_1(k) in general}
If $(p_k)$ is increasing then $T_1(n+1)\le 1 + n \cdot p_n$.
\end{theorem}
\begin{proof}
We use a similar argument as for Theorem~\ref{T_1(n) is quadratic}. Let $t=T_1(n)+1$.  After time $t$, each card will jump over card $1$ after being seen at most $n$ times in total. This means that for $1<i\le p_n$, card $i$ needs to be seen at most $n - c_i(t)$ more times. Using the relation
\[
\sum^{p_n}_{i=2}c_i(t) = T_1(n) - n + 1,
\]
we have
\[
T_1(n+1) - T_1(n) - 1 \le (p_n-1)n+1-(T_1(n)-n+1) = p_n\cdot n-T_1(n),
\]
and the result follows.
\end{proof}

\begin{remark}
The result is not true without the assumption that $(p_k)$ is increasing, since in the difference $T_1(k+1)-T_1(k)$, we only need to subtract the values $c_i(t)$ for cards that are before card $1$ in the deck.  However, if $(p_k)$ is not always increasing, then there may be some cards that have already been seen but that lie after card $1$; we should not add the $k - c_i(t)$ terms associated with these cards.
\end{remark}

For any (not necessarily increasing) sequence $(p_k)$, we can also prove analogues of Theorem~\ref{T_n(1) is quadratic} and \cite[Theorem 7]{NKS} (which shows $T_i(n+1) - T_i(n) \leq n^2$ for the slow flashcard game $p_k = k + 1$).  The next result gives an upper bound on $T_n(1)$; this result is stronger than the bound implied by Theorems~\ref{thm:interweaving in general} and~\ref{thm:T_1(k) in general} in the case that $(p_k)$ is increasing.
\begin{theorem}\label{thm:T_n(1) in general}
We have
\[
T_n(1) \leq 1 + (n - 1) \cdot \min\{j \colon p_j \geq n\}.
\]
\end{theorem}
\begin{proof}
Similar to Theorem~\ref{T_n(1) is quadratic}. All of the $n - 1$ cards in front of card $n$ must jump after card $n$. This happens after each card is seen at most $\min\{j \colon p_j \geq n\}$ times.
\end{proof}

%When $(p_k)$ is increasing and $n = p_k$, Theorem~\ref{thm:T_n(1) in general} gives the same upper bound on $T_n(1)$ as that given by Theorems~\ref{thm:T_1(k) in general} and~\ref{thm:interweaving in general}, but it's an improvement for other $n$.

We can also prove a general upper bound on the differences $T_i(n + 1) - T_i(n)$.

\begin{theorem}\label{thm:general difference bound}
For all $i$ and $n$ and all sequences $(p_j)$, we have
$T_i(n+1)-T_i(n)\le (p_n-1)n+1$.
\end{theorem}
\begin{proof}
Similar to \cite[Theorem 7]{NKS}. We want every one of the $p_n - 1$ cards in front of the card $i$ in position $p_n$ to jump over it. This happens after each card is seen at most $n$ times.
\end{proof}

\begin{problem}
Theorem~\ref{thm:general difference bound} is sharp when the sequence $(p_k)$ satisfies $p_k \mid p_{k + 1}$ for all $k$ \cite[the ``generalized recap schedule'']{NKS} but (assuming Conjecture~\ref{conj:NKS} holds) has room for improvement when $p_k = k + 1$.  Can we say anything when $p_k$ grows like a polynomial in $k$?  In particular, is the bound of Theorem~\ref{thm:general difference bound} always too lax in this case?
\end{problem}

\section{Open questions, generalizations and variations}\label{sec:open questions}

In this section, we consider several other variations and extensions on the notion of a flashcard game.  We do not seek to prove any major results, but rather to suggest possible directions for future research in addition to those conjectures and questions scattered throughout the preceding sections.

\subsection{Flashcard games}

Several intriguing problems related to the Slow Flashcard Game (e.g., Conjectures~\ref{stable} and~\ref{conj:NKS}) remain open.  These conjectures amount to particular aspects of the following general project.
\begin{problem}
Characterize the functions $T_n(k)$ (that give the time at which card $n$ is seen for the $k$th time) or equivalently $c_n(t)$ (that gives the number of times card $i$ has been seen at time $t$).
\end{problem}
Similarly, one can ask to understand these functions in the context of a general flashcard game.  For example, it seems natural to ask how good the results in Section~\ref{sec:generalizing} are.
\begin{problem}
When are the bounds in Section~\ref{sec:generalizing} tight?
\end{problem}

\subsection{Multiplication by permutations}
We can recast flashcard games as certain processes on the group $S_\infty$ of permutations of $\Z_{> 0}$ that fix all but finitely many values.  The operation ``move the card at front of the deck to the $p_k$th position'' is equivalent to multiplying the deck (thought of as a member of $S_\infty$ in one-line notation) by the $p_k$-cycle $C_{p_k}= (1,2,\dots,p_k)$.  This immediately suggests the following generalization: given a sequence $(\sigma_k)_{k \in \Z_{> 0}}$ of members of $S_\infty$ and starting with the deck in the usual order $1, 2, 3, \ldots$, upon viewing a card for the $k$th time, multiply by the deck by the permutation $\sigma_k$.  We mention four reasonable-seeming choices for the $\sigma_k$; two are easy to understand and not very interesting, while two behave in a more complicated fashion.
\begin{enumerate*}
\item If $\sigma_k = (1, k + 1)$ is the transposition that switches the cards in the first and $(k + 1)$th position, the associated viewing sequence is very simple:
\[
1,2,1,3,2,1,4,3,2,1,5,4,3,2,1,\dots.
\]
In the same way that Novikoff \emph{et al.} view certain sequences as reading ordered for labeled trees, this order can be realized as follows: in the tree
$$ \xy 0;/r.17pc/: (0,0)*{\boxed{1}}="a";
 (10,0)*{\boxed{2}}="b";
 (20,0)*{\boxed{3}}="c";
 (30,0)*{\boxed{4}}="d";
 (5,10)*{\boxed{1}}="e";
 (12,20)*{\boxed{21}}="f";
 (20,30)*{\boxed{321}}="g";
"a"; "e"**\dir{-};
"b"; "e"**\dir{-};
"f"; "e"**\dir{-};
"c"; "f"**\dir{-};
"g"; "f"**\dir{-};
"d"; "g"**\dir{-};
(30,20)*{\cdots}
\endxy$$
we start from the left-most leaf and go from the left leaf to the right leaf and then to their parent. This way we get the sequence above.

Note that in this example we have
\[
T_n(k)-T_n(k-1)=n+k-1,
\]
while an ideal learning process should have the property that $T_n(k)-T_n(k-1)<f(k)$, independent of $n$.
\item If $\sigma_k = (1, k + 1)(2, k)(3, k - 1) \cdots$ is the permutation that reverses the order of the first $k + 1$ cards in the deck, the resulting reading order is
\[
1, 2, 1, 3, 2, 1, 4, 2, 1, 5, 2, 1, 6, 2, 1, \ldots
\]
and only the cards $1$ and $2$ are viewed infinitely many times.
\item If $\sigma_k = (1, k + 1)(2, k + 2)(3, k + 3) \cdots (k, 2k)$ is a ``cut'' of the deck that switches the first $k$ cards with the next $k$, the viewing sequence is
\[
1, 2, 1, 3, 4, 3, 1, 4, 6, 1, 2, 7, 8, 7, 2, 8, 6, 8, 2, 1, 3, 9, 10, 9, 1, 6, 2, 5, 1, 12, 6, \ldots.
\]
\item If $\sigma_k$ is given in one-line notation as $\sigma_k = (k + 1)1(k + 2)2(k + 3)3\cdots (k - 1)(2k)k (2k + 1)(2k + 2) \cdots$ (i.e., it is the permutation that applies a shuffle to the first $2k$ cards), the viewing sequence is
\[
1, 2, 1, 3, 1, 2, 5, 2, 1, 4, 1, 6, 1, 9, 1, 4, 11, 4, 1, 3, 10, 3, 1, 2, 9, 13, 9, 2, 1, 16, 1, \ldots.
\]
\end{enumerate*}
The last two examples exhibit very mysterious behavior; it would be interesting to give any quantitative description of them.

\subsection{Permutation statistics}
We continue to view the deck of cards at time $t$ as an infinite permutation, but return to the case of flashcard games.  Since the deck has (as a permutation of $\Z_{> 0}$) only finitely many non-fixed points, many classical permutation statistics on $S_n$ make sense for this permutation.  For example, the permutation statistics $\inv$ and $\des$ (number of inversions and descents, respectively) are both well-defined.  The following question about the evolution of these statistics is natural.
\begin{problem}
What is the growth rate of the number of inversions of the deck as a function of time?  Descents?  Other interesting permutation statistics?
\end{problem}

\subsection{Randomness}
Finally, we observe that though our entire discussion has been about deterministic procedures, it is perhaps most natural to consider randomized versions of this process, e.g., to treat $p_k$ not as a fixed value but instead as a random variable with distribution depending on $k$.  (Observe that if $p_k$ does \emph{not} depend on $k$, i.e., if we insert the last-viewed card into a random position in the deck without regard for how many times the card has been viewed, then we are performing a random walk on the Cayley graph of $S_\infty$ generated by the cycles $C_{p_1}$, $C_{p_2}$, etc.)  Two natural choices for $(p_k)$ are the uniform distribution on $[1, 2k + 1]$ or a Poisson distribution with mean $k$.  What can be said about the behavior of the functions $T_n(k)$ in this case?

\section*{Acknowledgements}
\label{sec:ack}
The authors are grateful to Tim Novikoff and Ekaterina Orekhova for interesting discussions; to Mark Lipson for providing Theorem~\ref{T_n(1) is quadratic} and several very helpful comments; and to an anonymous referee for a variety of helpful suggestions.

\bibliographystyle{alpha}
\bibliography{FlashcardBibliography}
\label{sec:biblio}
\end{document}